\documentclass{amsart}
\begin{document}
%
%
%
 \newtheorem{thm}{Theorem}[section]
 \newtheorem{cor}[thm]{Corollary}
 \newtheorem{lem}[thm]{Lemma}
 \newtheorem{prop}[thm]{Proposition}
 \theoremstyle{definition}
 \newtheorem{defn}[thm]{Definition}
 \theoremstyle{remark}
 \newtheorem{rem}[thm]{Remark}
 \newtheorem*{ex}{Example}
 \numberwithin{equation}{section}
 \def\e{{\rm e}}
\def\F{\mathcal{F}}
\def\R{\mathbb{R}}
\def\T{\mathbf{T}}
\def\N{\mathbb{N}}
\def\K{\mathbf{K}}
\def\Q{\mathbf{Q}}
\def\M{\mathbf{M}}
\def\O{\mathbf{O}}
\def\C{\mathbb{C}}
\def\P{\mathbf{P}}
\def\Z{\mathbb{Z}}
\def\H{\mathcal{H}}
\def\A{\mathbf{A}}
\def\V{\mathbf{V}}
\def\AA{\overline{\mathbf{A}}}
\def\B{\mathbf{B}}
\def\c{\mathbf{C}}
\def\L{\mathbf{L}}
\def\S{\mathbb{S}}
\def\H{\mathcal{H}}
\def\I{\mathbf{I}}
\def\Y{{\rm Y}}
\def\f{\mathbf{f}}
\def\z{\mathbf{z}}
\def\d{\hat{d}}
\def\y{\mathbf{y}}
\def\w{\mathbf{w}}
\def\b{\mathcal{B}}
\def\s{\mathcal{S}}
\def\cc{\mathcal{C}}
\def\co{{\rm co}\,}
\def\vol{{\rm vol}\,}
\def\om{\mathbf{\Omega}}

%
%
%
%
%
%
%
%
%
\title{Certificates of convexity for basic semi-algebraic sets}

\author{JB. Lasserre}
\address{LAAS-CNRS and Institute of Mathematics\\
University of Toulouse\\
LAAS, 7 avenue du Colonel Roche\\
31077 Toulouse C\'edex 4\\
France}
\email{lasserre@laas.fr}
\thanks{This work was completed with the support of 
the (french) ANR grant NT05-3-41612.}



\date{}

\begin{abstract}
We provide two certificates of convexity
for arbitrary basic closed semi-algebraic sets of $\R^n$. The first one is based on a necessary and sufficient
condition whereas the second one is based on a sufficient (but simpler) condition only.
Both certificates are obtained from any feasible solution
of a related semidefinite program and so in principle, can be obtained numerically
(however, up to machine precision).
\end{abstract}

\maketitle

\section{Introduction}~

With $\R[x]$ being the ring of real polynomials in the variables $x_1,\ldots,x_n$, 
consider the basic closed semi-algebraic set $\K\subset\R^n$ defined by:
\begin{equation}
\label{setk}
\K:=\,\{x\in\R^n\::\:g_j(x)\geq0,\quad j=1,\ldots,m\}
\end{equation}
for some given polynomials $g_j\in\R[x]$, $j=1,\ldots,m$. 

By definition, $\K\subset\R^n$ is convex if and only if
\begin{equation}
\label{usual}
x,\,y\in\K\quad\Rightarrow\quad \lambda\, x+(1-\lambda)\, y\in\K \quad\forall \lambda\in [0,1].
\end{equation}
The above geometric condition does {\it not} depend on the representation of $\K$ but 
requires uncountably many tests and so cannot be checked in general. 

Of course concavity of $g_j$ for every $j=1,\ldots,m$, provides a certificate of convexity for $\K$ 
but not every convex set $\K$ in (\ref{setk}) is defined by concave polynomials.
Hence an important issue is to analyze whether 
there exists a necessary and sufficient condition of convexity in terms of the representation (\ref{setk}) of
$\K$ because after all, very often (\ref{setk}) is the only information available about $\K$. 
Moreover, a highly desirable feature would be that such a condition 
can be checked, at least numerically.

In a recent work \cite{lasserreconvex}, the author has provided an algorithm to obtain a numerical 
{\it certificate} of convexity for $\K$ in (\ref{setk}) by using the condition:
\begin{equation}
\label{old}
\langle \nabla g_j(y),x-y\rangle\,\geq\,0,\qquad \forall x,y\in\K \mbox{ with }g_j(y)=0,\end{equation}
which is equivalent to (\ref{usual}) provided that Slater\footnote{Slater condition holds if there exists $x_0\in\K$ such that $g_k(x_0)>0$ for every $k=1,\ldots,m$.} condition holds and the nondegeneracy condition
$\nabla g_j(y)\neq0$ holds whenever $y\in \K$ and $g_j(y)=0$. 
This certificate consists of an integer $p_j$ and 
two polynomials $\theta^j_1,\theta^j_2\in\R[x,y]$ for each $j=1,\ldots,m$, and their characterization obviously implies that (\ref{old}) holds true and so $\K$ is convex (whence the name certificate); see Lasserre \cite[Corollary 4.4]{lasserreconvex}.
More precisely, for every $j=1,\ldots,m$, define the $2m+1$ polynomials 
$h_\ell\in\R[x,y]$ by $h_\ell(x,y)=g_\ell(x)$, $h_{m+\ell}(x,y)=g_\ell(y)$
for every $l=1,\ldots,m$, and $h_{2m+1}(x,y)=-g_j(y)$.
The {\it preordering} $P_j\subset\R[x,y]$ generated by the
polynomials $(h_\ell)\subset\R[x,y]$
is defined by:
\begin{equation}
\label{preordering1}
P_j\,=\left\{ \sum_{J\subseteq\{1,\ldots,2m+1\}}
\sigma_J(x,y)\left(\prod_{\ell\in J}h_\ell(x,y)\right)\::\:\quad \sigma_J\in\Sigma[x,y]\,\right\},\end{equation}
where $\Sigma[x,y]\subset\R[x,y]$ is the set of polynomials that are {\it sums of squares}
(in short s.o.s.), and  where by convention, $\prod_{\ell\in J} h_\ell(x,y)=1$ when $J=\emptyset$.
Then by a direct application of 
Stengle's Positivstellensatz \cite[Theor. 4.4.2, p. 92]{roy}
(more precisely, a Nichtnegativstellensatz version)
(\ref{old}) holds if and only if
\begin{equation}
\label{neweq1}
\theta^j_1(x,y)\langle \nabla g_j(y),x-y\rangle =(\langle \nabla g_j(y),x-y\rangle)^{2p_j}
+\theta_2^j(x,y),\end{equation}
for  some integer $p_j$ and some polynomials $\theta^j_1,\theta^j_2\in P_j$.
In addition, bounds $(p,d)$ are available
for the integer $p_j$ and the degrees of the s.o.s. polynomials $\sigma_J$ 
 appearing in the definition (\ref{preordering1}) of polynomials $\theta^j_1,\theta^j_2\in P_j$, respectively.
Observe that in (\ref{neweq1}) one may replace $p_j$ with the fixed bound $p$
(multiply each side with $(\langle \nabla g_j(y),x-y\rangle)^{2(p-p_j)}$) and take $d:=d+p$.
Next, recall that s.o.s. polynomials of bounded degree can be obtained
from feasible solutions of an appropriate {\it semidefinite program}\footnote{
A semidefinite program is a convex optimization problem
with the nice property that it can be solved efficiently. More precisely, up to arbitrary fixed precision, it can be solved in time polynomial in its input size.
For more details on semidefinite programming and its applications, the interested reader is referred to e.g. \cite{boyd}.}
(see e.g. \cite{lasserresiopt}).
Hence, in principle, checking whether (\ref{neweq1}) has a feasible
solution $(\theta^j_1,\theta^j_2)$ reduces to checking whether a {\it single} semidefinite program has a feasible solution.

And so, when both Slater and the nondegeneracy condition hold,
checking whether $\K$ is convex reduces to checking if each of the semidefinite programs
associated with (\ref{neweq1}), $j=1,\ldots,m$,
has a feasible solution.
When $\K$ is convex, the $2m$ polynomials
$\theta^j_1,\theta^j_2\in\R[x,y]$, $j=1,\ldots,m$, provide the desired certificate
of convexity through (\ref{old}); see \cite[Corollary 4.4]{lasserreconvex}.
However, it is only a {\it numerical} certificate because it comes
from the output of a numerical algorithm, and so subject to anavoidable numerical inaccuracies.
Moreover, the size of each semidefinite program equivalent to (\ref{neweq1})
is out of reach for practical computation, and in practice, one will
solve a semidefinite program associated with (\ref{neweq1}) but
for reasonable bounds $(p',d')\ll (p,d)$, hoping to obtain a solution when $\K$ is convex.
An alternative and more tractable certificate of convexity
using {\it quadratic modules} rather than preorderings 
is also provided in \cite[Assumption 4.6]{lasserreconvex}, but it only provides
 a sufficient condition of convexity (almost necessary when $\K$
 is compact and satisfies some technical condition).
  
The present contribution is to provide a certificate of convexity for 
{\it arbitrary} basic closed semi-algebraic sets (\ref{setk}), i.e., with {\it no} assumption on $\K$. This time,
by certificate we mean an obvious guarantee that
the geometric condition (\ref{usual}) holds true (instead of (\ref{old}) in \cite{lasserreconvex}).
To the best of our knowledge, and despite the result is almost straightforward, it is the first 
of this type for arbitrary basic closed semi-algebraic sets.
As in \cite{lasserreconvex} our certificate also consists of 
two polynomials of $\R[x,y]$ 
and is also based on the powerful Stengle's Positivstellensatz in real algebraic geometry.
In addition, a numerical certificate can also be obtained as the output of a
semidefinite program (hence valid only up to machine precision).
We also provide another certificate based on a simpler
characterization which now uses only a sufficient condition 
for a polynomial to be nonnegative on $\K$; so in this case, even if $\K$ is convex, 
there is no guarantee to obtain the required certificate.
Finally, we also provide a sufficient condition that permits to
obtain a numerical certificate of non convexity of $\K$ in the form of 
points $x,y\in\K$ which violate (\ref{usual}).

\section{Main result}

Observe that in fact, (\ref{usual}) is equivalent to the simpler condition
\begin{equation}
\label{usual1}
x,\,y\in\K\quad\Rightarrow\quad\,(x+y)/2\in\K.
\end{equation}
Indeed if $\K$ is convex then of course (\ref{usual1}) holds.
Conversely, if $\K$ is not convex then there exists $x,y\in\K$ and $0<\lambda<1$ such
that $z:=x+\lambda (y-x)$ is not in $\K$. As $\K$ is closed,
moving on the line segment $[x,y]$ from $x$ to $y$,
there necessarily exist $\tilde{x}\in\K$ (the first exit point of $\K$)
and $\tilde{y}\in\K$ (the first re-entry point in $\K$), with $\tilde{x}\neq\tilde{y}$.
Thus, as $\tilde{x}$ and $\tilde{y}$ are the only points of
$[\tilde{x},\tilde{y}]$ contained in $\K$, the mid-point
$\tilde{z}:=(\tilde{x}+\tilde{y})/2$ is not contained in $\K$. 

Given the basic closed semi-algebraic set $\K$ defined in (\ref{setk}), 
let $\widehat{\K}:=\K\times\K\subset\R^n\times\R^n$ be the associated basic closed semi-algebraic set
defined by:
\begin{equation}
\label{hatk}
\widehat{\K}\,:=\,\{(x,y)\::\:\hat{g}_j(x,y)\geq0,\:j=1,\ldots,2m\:\},
\end{equation}
where: \begin{eqnarray}
\label{a1}
(x,y)\,\mapsto\,\hat{g}_j(x,y)&:=&g_j(x),\quad j=1,\ldots,m\\
\label{a2}
(x,y)\,\mapsto\,\hat{g}_{j}(x,y)&:=&g_{m-j}(y),\quad j=m+1,\ldots,2m,
\end{eqnarray}
and let $P(\hat{g})\subset\R[x,y]$ be the preordering associated with the polynomials 
$(\hat{g}_j)$
that define $\widehat{\K}$ in (\ref{hatk}), i.e.,
\begin{equation}
\label{preordering}
P(\hat{g})\,:=\,\left\{\sum_{J\subseteq\{1,\ldots,2m\}}\phi_J\,\left(\prod_{k\in J}\hat{g}_k\right)\::\quad \phi_J\in\Sigma[x,y]\:\right\},
\end{equation}
where $\Sigma[x,y]\subset\R[x,y]$ is the set of s.o.s. polynomials.
Our necessary and sufficient condition of convexity is a follows.

\begin{thm}
\label{th1}
Let $\K\subset\R^n$ be the basic closed semi-algebraic set defined in (\ref{setk}). Then
$\K$ is convex if and only if for every $j=1,\ldots,m$, there exist
polynomials $\sigma_j,h_j\in P(\hat{g})$ and an integer $p_j\in\N$ such that:
\begin{equation}
\label{th1-1}
\sigma_j(x,y)\,g_j((x+y)/2)\,=\,g_j((x+y)/2)^{2p_j}+h_j(x,y),\quad
\forall \,x,y\in\R^n.
\end{equation}

\end{thm}
\begin{proof}
The set $\K$ is convex if and only if (\ref{usual1}) holds, that is, if and only if for every $j=1,\ldots,m$,
\begin{equation}
\label{test}
g_j((x+y)/2)\geq 0\qquad\forall \,(x,y)\in\widehat{\K}.
\end{equation}
But then (\ref{th1-1}) follows from a direct application of 
Stengle's Positivstellensatz \cite[Theor. 4.4.2, p. 92]{roy}
to (\ref{test}) (in fact, a Nichtnegativstellensatz version).
\end{proof}
The polynomials $\sigma_j,h_j\in P(\hat{g})$, $j=1,\ldots,m$, obtained in (\ref{th1-1}) indeed provide an obvious certificate of convexity for $\K$. This is because
if (\ref{th1-1}) holds then for every $x,y\in\K$ one has
$\sigma_j(x,y)\geq0$ and $h_j(x,y)\geq0$ because $\sigma_j,h_j\in P(\hat{g})$; and so
$\sigma_j(x,y)g_j((x+y)/2)\geq0$.
Therefore if $\sigma_j(x,y)>0$ then $g_j((x+y)/2)\geq0$ whereas
if $\sigma_j(x,y)=0$ then
$g_j((x+y)/2)^{2p_j}=0$ which in turn implies
$g_j((x+y)/2)=0$. Hence for every $j=1,\ldots,m$, 
$g_j((x+y)/2)\geq0$ for every $x,y\in\K$, that is,
(\ref{usual1}) holds and so $\K$ is convex.

\subsection*{A numerical certificate of convexity}
Again, as (\ref{th1-1}) is coming from Stengle's Positivstellensatz,
bounds $(p,d)$ are available for the integer $p_j$ and the degrees of 
the s.o.s. polynomials $\phi_J$ 
 appearing in the definition (\ref{preordering}) of polynomials 
 $\sigma_j,h_j\in P(\hat{g})$, respectively. Hence,
with same arguments as in the discussion just after (\ref{neweq1}),
checking whether (\ref{th1-1}) holds
reduces to check whether some (single)
approriately defined semidefinite program has a feasible solution.

Hence checking convexity of the basic closed semi-algebraic set 
$\K$ reduces to checking whether
each semidefinite program associated with (\ref{th1-1}), $j=1,\ldots,m$, 
has a feasible solution, and any feasible solution $\sigma_j,h_j\in P(\hat{g})$ of (\ref{th1-1}), $j=1,\ldots,m$, provides a certificate of convexity for $\K$. However the certificate is only "numerical" as the coefficients of the polynomials
$\sigma_j,h_j$ are obtained numerically and are subject to anavoidable
numerical inaccuracies. Moreover, the bounds $(p,d)$ being
out of reach, in practice one will
solve a semidefinite program associated with (\ref{th1-1}) but
for reasonable bounds $(p',d')\ll (p,d)$, hoping to obtain a solution when $\K$ is convex.

\subsection{An easier sufficient condition for convexity}
While Theorem \ref{th1} provides a necessary and sufficient condition for convexity, it is very expensive to check
because for each $j=1,\ldots,m$, the certificate of convexity $\sigma_j,h_j\in P(\hat{g})$ in (\ref{th1-1}) involves computing
$2\times 2^{2m}=2^{2m+1}$ s.o.s. polynomials $\phi_J$ in the definition
(\ref{preordering}) of $\sigma_j$ and $h_j$ . However, one also has the following sufficient condition:
\begin{thm}
\label{th2}
Let $\K\subset\R^n$ be the basic semi-algebraic set defined in (\ref{setk}). Then
$\K$ is convex if for every $j=1,\ldots,m$:
\begin{equation}
\label{th2-1}
g_j((x+y)/2)=\sigma_0(x,y)+
\sum_{k=1}^m\sigma^j_k(x,y) g_k(x)+\psi^j_k(x,y)\,g_k(y),\quad \forall x,y\in\R^n,
\end{equation}
for some s.o.s. polynomials $\sigma^j_k,\psi^j_k\in \Sigma[x,y]$.
\end{thm}
\begin{proof}
Observe that if (\ref{th2-1}) holds then 
$g_j((x+y)/2)\geq0$ for every $j=1,\ldots,m$ and all $(x,y)\in\widehat{\K}$;
and so $\K$ is convex because (\ref{usual1}) holds.
\end{proof}
Again, checking whether (\ref{th2-1}) holds with an apriori bound $2d$ on the degrees of the s.o.s.
polynomials $\sigma^j_k,\psi^j_k$, reduces to solving a semidefinite program. But it now only involves $2m+1$ unknown
s.o.s. polynomials (to be compared with $2^{2m+1}$ previously). On the other hand, 
Theorem \ref{th2} only provides a sufficient condition, that is, even if $\K$ is convex it may happen that
(\ref{th2-1}) does not hold.

However, when $\K$ is compact, convex, and if for some $M>0$ the quadratic polynomial
$x\mapsto M-\Vert x\Vert^2$ can be written
\begin{equation}
\label{quadratic}
M-\Vert x\Vert^2 \,=\,\sigma_0(x)+\sum_{k=1}^m\sigma_k(x)\,g_j(x),\end{equation}
for some s.o.s. polynomials $(\sigma_k)\subset\Sigma[x]$, then
(\ref{th2-1}) is {\it almost} necessary because for every $\epsilon>0$:
\begin{equation}
\label{final}
g_j((x+y)/2)+\epsilon=\sigma^j_{0\epsilon}(x,y)
+\sum_{k=1}^m\sigma^j_{k\epsilon}(x,y) g_k(x)+\psi^j_{k\epsilon}(x,y)\,g_k(y),
\end{equation}
for some s.o.s. polynomials $\sigma^j_{k\epsilon},\psi^j_{k\epsilon}\in \Sigma[x,y]$.
Indeed, consider the quadratic polynomial
\[(x,y)\mapsto \Delta(x,y)\,:=\,2M-\Vert x\Vert^2-\Vert y\Vert^2.\]
From (\ref{quadratic}), $\Delta$ belongs to the quadratic module
$Q(\hat{g})\subset\R[x,y]$ generated by the polynomials $\hat{g}_k$ that define $\widehat{\K}$, that is, the set
\[Q(\hat{g}):=
\left\{\sigma_0(x,y)+\sum_{k=1}^m\sigma_k(x,y)\,g_k(x)+\psi_k(x,y)\,g_k(y)\::\:
\sigma_k,\psi_k\in\Sigma[x,y]\right\}.\]
In addition, its level set $\{(x,y)\,:\,\Delta(x,y)\geq0\}$ is compact,
which implies that $Q(\hat{g})$ is  Archimedean (see e.g. \cite{markus}).
Therefore, as $g_j((x+y)/2)+\epsilon>0$ on $\widehat{\K}$,
(\ref{final}) follows from Putinar's Positivstellensatz \cite{putinar}.

\subsection{A certificate on non-convexity}

In this final section we provide a numerical certificate of {\it non convexity} of $\K$
when the optimal value of a certain semidefinite program is strictly negative 
and some moment matrix associated with an optimal solution satisfies
a certain rank condition.

Given a sequence $\z=(z_{\alpha\beta})$ indexed in the canonical basis
$(x^\alpha y^\beta)$ of $\R[x,y]$, let $L_\z:\R[x,y]\to\R$ be the linear functional:
\[f\quad(=\sum_{\alpha,\beta}f_{\alpha\beta}\,x^\alpha y^\beta)\:\mapsto\:L_\z(f)\,=\,
\sum_{\alpha,\beta}f_{\alpha\beta}\,z_{\alpha \beta},\]
and as in \cite{lasserresiopt}, 
the {\it moment} matrix $M_s(\z)$ associated with $\z$ is the real symmetric matrix
with rows and columns indexed in the
the canonical basis $(x^\alpha y^\beta)$ and with entries
\[M_s(\z)((\alpha,\beta),(\alpha',\beta'))\,=\,z_{(\alpha+\alpha')(\beta+\beta')}\]
for every $(\alpha,\beta),(\alpha',\beta')\in\N^{2n}_s$, where
$\N^{n}_s:=\{\alpha\in\N^n\,:\,\sum_i\alpha_i\leq s\}$.

Similarly, with a polynomial $(x,y)\mapsto \theta(x,y)=\sum_{\alpha,\beta}
\theta_{\alpha\beta}\,x^\alpha y^\beta$, the {\it localizing} matrix $M_s(\theta\,\z)$
associated with $\theta$ and $\z$, is the real symmetric matrix with rows and columns indexed in the
the canonical basis $(x^\alpha y^\beta)$ and with entries
\[M_s(\theta\,\z)((\alpha,\beta),(\alpha',\beta'))\,=\,\sum_{\alpha",\beta"}
\theta_{\alpha"\beta" }\,z_{(\alpha+\alpha'+\alpha")(\beta+\beta'+\beta")},\]
for every $(\alpha,\beta),(\alpha',\beta')\in\N^{2n}_s$.

Let $v_k:=\lceil({\rm deg}\,\hat{g}_k)/2\rceil$, $k=1,\ldots,2m$, and
for every $j=1,\ldots,m$, and $s\geq v:=\max_kv_k$, consider the 
semidefinite program:
\begin{equation}
\label{sdp}
\left\{\begin{array}{lll}
\rho_{js}=&\displaystyle\min_\z&L_\z(g_j((x+y)/2))\\
&\mbox{s.t.}&M_s(\z)\,\succeq0\\
&&M_{s-v_k}(\hat{g}_k\,\z)\,\succeq0,\quad k=1,\ldots,2m\\
&&\z_0\,=1,\end{array}\right.\end{equation}
where for a real symmetric matrix $A$, 
the notation $A\succeq0$ stands for $A$ is positive semidefinite.
The semidefinite program (\ref{sdp}) is a convex relaxation of the global optimization problem
\[g_j^*:=\min_{x,y}\:\{g_j((x+y)/2)\::\:(x,y)\in\widehat{\K}\:\}\]
and so $\rho_{js}\leq g_j^*$ for every $s\geq v$. Moreover, $\rho_{js}\uparrow g_j^*$ as $s\to\infty$; for more details see e.g. \cite{lasserresiopt}.
\begin{thm}
\label{th3}
Let $\K\subset\R^n$ be as in (\ref{setk}) and let $\z$ be an optimal solution of the semidefinite program
(\ref{sdp}) with optimal value $\rho_{js}$. If $\rho_{js}<0$ and
\begin{equation}
\label{th3-1}
{\rm rank}\,M_s(\z)\,=\,{\rm rank}\,M_{s-v}(\z)\quad (=:t)
\end{equation}
then the set $\K$ is not convex and one may extract $t$ points
$(x(i),y(i))\in\widehat{\K}$, $i=1,\ldots,t$, such that 
\[g_j((x(i)+y(i))/2)\,<\,0,\qquad \forall\,i=1,\ldots,t.\]
Hence each mid-point $(x(i)+y(i))/2\not\in\widehat{\K}$ is a certificate that $\K$ is not convex.
\end{thm}
\begin{proof}
By the flat extension theorem of Curto and Fialkow \cite{curto} (see also Laurent \cite{laurent}), the rank condition (\ref{th3-1}) ensures that
$\z$ is the moment sequence of a $t$-atomic probability measure $\mu$ supported on $\widehat{\K}$. That is:
\[z_{\alpha\beta }\,=\,\int_{\widehat{\K}} x^\alpha\,y^\beta\,d\mu,\qquad
\forall\,(\alpha,\beta)\in\N^{2n}_{2s}.\]
Let $(x(i),y(i))_{i=1}^t\subset\widehat{\K}$ be the support of $\mu$ which
is a positive linear combination of Dirac measures $\delta_{(x(i),y(i))}$
with positive weights $(\gamma_i)$ such that $\sum_i\gamma_i=1$. Then
\begin{eqnarray*}
g_j^*\geq \rho_{js}=L_\z(g_j((x+y)/2))&=&\int_{\widehat{\K}}
g_j((x+y)/2)\,d\mu\\
&=&\sum_{i=1}^t\gamma_i\,g_j((x(i)+y(i))/2)\\
&\geq&\sum_{i=1}^t\gamma_i\,g_j^*\,=\,g_j^*,
\end{eqnarray*}
which shows that $\rho_{js}=g_j^*$ and so,
$g_j((x(i)+y(i))/2)=g_j^*$ for every $i=1,\ldots,t$. But then
the result follows from $\rho_{js}<0$.
\end{proof}

\end{document}